\documentclass[reqno]{amsart}

\usepackage[latin1]{inputenc}
\usepackage{amssymb}
\usepackage{graphicx}
\usepackage{amscd}
\usepackage[hidelinks]{hyperref}
\usepackage{color}
\usepackage{float}
\usepackage{graphics,amsmath,amssymb}
\usepackage{amsthm}
\usepackage{amsfonts}
\usepackage{latexsym}
\usepackage{epsf}
\usepackage{enumerate}
\usepackage{xifthen}
\usepackage{mathrsfs}
\usepackage{dsfont}
\usepackage{makecell}
\usepackage[FIGTOPCAP]{subfigure}
\usepackage{amsmath}
\allowdisplaybreaks[4]
\usepackage{listings}
\usepackage{etoolbox}
\usepackage{fancyhdr}
\usepackage{pdflscape}
\usepackage[title,toc,titletoc]{appendix}

 \newtheoremstyle{mytheorem}
 {3pt}
 {3pt}
 {\slshape}
 {}
 {\bfseries}
 {.}
 { }
 {}

\numberwithin{equation}{section}

\theoremstyle{theorem}
\newtheorem{theorem}{Theorem}[section]
\newtheorem{corollary}[theorem]{Corollary}
\newtheorem{lemma}[theorem]{Lemma}

\theoremstyle{definition}

\newtheorem{remark}{Remark}[section]

\newcommand{\Keywords}[1]{\ifthenelse{\isempty{#1}}{}{\smallskip \smallskip \noindent \textbf{Keywords}. #1}}
\newcommand{\MSC}[2][2010]{\ifthenelse{\isempty{#2}}{}{\smallskip \smallskip \noindent \textbf{#1MSC}. #2}}
\newcommand{\abstractnote}[1]{\ifthenelse{\isempty{#1}}{}{\smallskip \smallskip \noindent \textsuperscript{\dag}#1}}

\makeatletter
\def\specialsection{\@startsection{section}{1}%
  \z@{\linespacing\@plus\linespacing}{.5\linespacing}%
  {\normalfont}}
\def\section{\@startsection{section}{1}%
  \z@{.7\linespacing\@plus\linespacing}{.5\linespacing}%
  {\normalfont\scshape}}
\patchcmd{\@settitle}{\uppercasenonmath\@title}{\Large\boldmath}{}{}
\patchcmd{\@settitle}{\begin{center}}{\begin{flushleft}}{}{}
\patchcmd{\@settitle}{\end{center}}{\end{flushleft}}{}{}
\patchcmd{\@setauthors}{\MakeUppercase}{\normalsize}{}{}
\patchcmd{\@setauthors}{\centering}{\raggedright}{}{}
\patchcmd{\section}{\scshape}{\large\bfseries\boldmath}{}{}
\patchcmd{\subsection}{\bfseries}{\bfseries\boldmath}{}{}
\renewcommand{\@secnumfont}{\bfseries}
\patchcmd{\@startsection}{\@afterindenttrue}{\@afterindentfalse}{}{}
\patchcmd{\abstract}{\leftmargin3pc}{\leftmargin1pc}{}{}

\def\maketitle{\par
  \@topnum\z@ 
  \@setcopyright
  \thispagestyle{empty}
  \ifx\@empty\shortauthors \let\shortauthors\shorttitle
  \else \andify\shortauthors
  \fi
  \@maketitle@hook
  \begingroup
  \@maketitle
  \toks@\@xp{\shortauthors}\@temptokena\@xp{\shorttitle}%
  \toks4{\def\\{ \ignorespaces}}
  \edef\@tempa{%
    \@nx\markboth{\the\toks4
      \@nx\MakeUppercase{\the\toks@}}{\the\@temptokena}}%
  \@tempa
  \endgroup
  \c@footnote\z@
  \@cleartopmattertags
}
\makeatother

\newcommand{\f}[1]{\ifthenelse{\equal{#1}{1}}{(q;q)_\infty}{(q^{#1};q^{#1})_{\infty}}}

\newcommand{\ee}[1]{\ifthenelse{\equal{#1}{1}}{E(q)}{E(q^{#1})}}
\newcommand{\fl}[1]{\left\lfloor#1\right\rfloor}
\newcommand{\bb}{\mathfrak{b}}
\title[Congruences arising from a second order mock theta function]{An infinite family of congruences arising from a second order mock theta function}

\author[S. Chern]{Shane Chern}
\address[S. Chern]{Department of Mathematics, The Pennsylvania State University, University Park, PA 16802, USA}
\email{shanechern@psu.edu}

\author[C. Wang]{Chun Wang}
\address[C. Wang]{Department of Mathematics, East China Normal University, 500 Dongchuan Road, Shanghai 200241, PR China}
\email{wangchunmath@outlook.com}

\date{}

\begin{document}

%

\maketitle

\begin{abstract}
Let $\beta(q)=\sum_{n\ge 0} \mathfrak{b}(n)q^n$ be a second order mock theta function defined by
$$\sum_{n\ge 0}\frac{q^{n(n+1)}(-q^2;q^2)_n}{(q;q^2)_{n+1}^2}.$$
In this paper, we obtain an infinite family of congruences modulo powers of $3$ for $\mathfrak{b}(n)$.


\MSC{11P83, 05A17.}
\end{abstract}

\section{Introduction}

In his last letter to Hardy \cite[pp.~220--223]{BR1995}, Ramanujan defined 17 functions, which he called mock theta functions. These functions subsequently attract the interest of many mathematicians. The interested readers may refer to \cite{BM2012} for a nice survey.

Let the $q$-shifted factorials be
\begin{align*}
(a)_n=(a;q)_n&:=\prod_{k=0}^{n-1}(1-aq^{k}),\\
(a)_\infty=(a;q)_{\infty}&:=\prod_{k\ge 0} (1-aq^{k}),\\
(a_1,a_2,\cdots,a_m;q)_\infty&:=(a_1;q)_\infty(a_2;q)_\infty\cdots(a_m;q)_\infty.
\end{align*}

One of the second order mock theta functions is defined by
\begin{equation}\label{B2}
\beta(q):=\sum_{n\ge 0}\frac{q^{n(n+1)}(-q^2;q^2)_n}{(q;q^2)_{n+1}^2}=\sum_{n\ge 0}\frac{q^n (-q;q^2)_n}{(q;q^2)_{n+1}}.
\end{equation}
According to Andrews \cite{And1981}, this function does not appear in Ramanujan's Lost Notebook. However, Andrews \cite{And1981} also claimed that $\beta(q)$ may connect with another second order mock theta function $\mu(q)$ shown in the Lost Notebook as well as Mordell integrals through identities such as
$$q^{-1/8}\sqrt{\frac{\pi}{2\alpha}}\ \mu(q)=\frac{2\pi}{\alpha}q_1^{1/2}\beta(q_1)+\int_{-\infty}^\infty \frac{e^{-2\alpha x^2+\alpha x}}{1+e^{2\alpha x}}\ dx,$$
where $q=e^{-\alpha}$, $q_1=e^{-\alpha'}$,  $\alpha\alpha'=\pi^2$, and
$$\mu(q):=\sum_{n\ge 0}\frac{(-1)^n q^{n^2}(q;q^2)_n}{(-q^2;q^2)_n^2}.$$

Using Watson's $q$-analog of Whipple's theorem \cite[p.~100, Eq.~(3.4.1.5)]{Sla1966}, Andrews \cite{And1981} showed that
$$\beta(q)=\frac{(-q^2;q^2)_\infty}{(q^2;q^2)_\infty}\sum_{n=-\infty}^\infty \frac{(-1)^n q^{2n(n+1)}}{1-q^{2n+1}}.$$
From the above representation of $\beta(q)$, Gordon and McIntosh proved (cf.~\cite[p.~136]{BM2012})
\begin{equation}\label{even}
\frac{\beta(q)+\beta(-q)}{2}=\frac{(q^4;q^4)_\infty^5}{(q^2;q^2)_{\infty}^4}.
\end{equation}

Let
\begin{equation*}
\sum_{n\ge 0}\bb(n) q^n = \beta(q).
\end{equation*}
It follows that
\begin{equation}\label{eq:a2n}
\sum_{n\ge 0}\bb(2n) q^n = \frac{(q^2;q^2)_\infty^5}{(q;q)_{\infty}^4}.
\end{equation}

We notice that many Ramanujan-type congruences for mock theta functions (or functions related to mock theta functions) are derived by dissecting the function to a $q$-series quotient (i.e.~product or quotient of certain $q$-shifted factorials). The interested readers may refer to \cite{APSY2017,CGH2018,Wan2017} for details. Since the r.h.s.~of \eqref{eq:a2n}, which comes from the 2-dissection of $\beta(q)$, is also a $q$-series quotient, it is natural to consider the arithmetic properties of $\beta(q)$. In this paper, we shall show the following infinite family of congruences modulo powers of $3$.

\begin{theorem}\label{th:cong}
For $\alpha\ge 1$ and $n\ge 0$, we have
\begin{equation}\label{eq:cong}
\bb\left(2\cdot 3^{2\alpha-1}n+\frac{3^{2\alpha}-1}{2}\right)\equiv 0 \pmod{3^{2\alpha}}.
\end{equation}
\end{theorem}

Our approach is an elementary refinement of the Watson--Atkin style proof \cite{Atk1967,Wat1938} (which involves modular forms) of congruences modulo powers of a prime for some $q$-series quotient. On the one hand, with the aid of some 3-dissection identities for Ramanujan's theta functions, we are able to provide elementary proofs of several initial relations (see our Theorems \ref{th:UX} and \ref{th:UxiX}), and hence avoid using modular forms. On the other hand, the recurrence relation shown in the Watson--Atkin style proof merely relies on Newton's identity. Combining the two ingredients together, we therefore arrive at a completely elementary proof of Theorem \ref{th:cong}.

\section{The infinite family of identities}

\subsection{Ramanujan's theta functions and 3-dissections}

For notational convenience, we write $E(q):=(q;q)_\infty$ throughout this paper.

Let $\phi(q)$ and $\psi(q)$ be two of Ramanujan's theta functions given by
$$\phi(q):=\sum_{n=-\infty}^\infty q^{n^2} \text{ and } \psi(q):=\sum_{n\ge 0}q^{n(n+1)/2}.$$
It is well known that
$$\phi(-q)=\frac{E(q)^2}{E(q^2)} \text{ and } \psi(q)=\frac{E(q^2)^2}{E(q)}.$$

We have the following 3-dissection identities.

\begin{lemma}\label{le:3-dis-1}
It holds that
\begin{align}
\frac{1}{\phi(-q)}&=\frac{\phi(-q^9)^3}{\phi(-q^3)^4}\left(1+2qw(q^3)+4q^2 w(q^3)^2\right),\label{eq:1/-phi--3dis}\\
\psi(q)&=\psi(q^9)\left(\frac{1}{w(q^3)}+q\right),\label{eq:psi-3--dis}
\end{align}
where
\begin{equation*}
w(q)=\frac{E(q)E(q^6)^3}{E(q^2)E(q^3)^3}.
\end{equation*}
\end{lemma}

One may refer to \cite[Eq.~(14.3.5)]{Hir2017} for \eqref{eq:psi-3--dis}. Let $\omega$ be a cube root of unity other than $1$. We substitute $\omega q$ and $\omega^2 q$ for $q$ in the 3-dissection of $\phi(-q)$ (cf.~\cite[Eq.~(14.3.4)]{Hir2017})
$$\phi(-q)=\phi(-q^9)\left(1-2qw(q^3)\right)$$
and multiply the two results to get \eqref{eq:1/-phi--3dis}.

We further notice that $q w(q)^3$ can be represented as follows.

\begin{lemma}\label{le:qw3}
It holds that
\begin{equation}
q w(q)^3=\frac{1}{8}\left(1-\frac{\phi(-q)^4}{\phi(-q^3)^4}\right).
\end{equation}
\end{lemma}

\begin{proof}
We have
\begin{align*}
qw(q)^3=q\left(\frac{E(q)E(q^6)^3}{E(q^2)E(q^3)^3}\right)^3=q\frac{\phi(-q)}{\phi(-q^3)^3}\frac{\psi(q^3)^3}{\psi(q)}.
\end{align*}
Recall the following identity (cf.~\cite[Eq.~(3.2)]{She1994})
$$\frac{\psi(q^3)^3}{\psi(q)}=\frac{1}{8q}\left(\frac{\phi(-q^3)^3}{\phi(-q)}-\frac{\phi(-q)^3}{\phi(-q^3)}\right).$$
Hence
\begin{align*}
qw(q)^3&=q\frac{\phi(-q)}{\phi(-q^3)^3}\left(\frac{1}{8q}\left(\frac{\phi(-q^3)^3}{\phi(-q)}-\frac{\phi(-q)^3}{\phi(-q^3)}\right)\right)\\
&=\frac{1}{8}\left(1-\frac{\phi(-q)^4}{\phi(-q^3)^4}\right).
\end{align*}
\end{proof}

\subsection{Initial relations}

Let $q:=\exp(2\pi i\tau)$ with $\tau\in\mathbb{H}$, the upper half complex plane. We put
\begin{align*}
X=X(\tau)&:=\frac{\ee{2}^4 \ee{3}^8}{\ee{1}^8 \ee{6}^4}=\frac{\phi(-q^3)^4}{\phi(-q)^4},\\
\xi=\xi(\tau)&:=q^{-2}\frac{\ee{2}^5 \ee{9}^4}{\ee{1}^4 \ee{18}^5}=q^{-2}\frac{\psi(q)^2 \phi(-q^9)}{\phi(-q) \psi(q^9)^2}.
\end{align*}

For a $q$-series expansion $\sum_{n\ge n_0} a_n q^n$, we introduce the $U$-opeartor defined by
$$U\left(\sum_{n\ge n_0} a_n q^n\right):=\sum_{3n\ge n_0} a_{3n} q^n.$$

We shall show
\begin{theorem}\label{th:UX}
It holds that
\begin{align}
U(X)&=10X- 36X^2+ 27X^3,\\
U(X^2)&=-8X+306X^2-2160X^3+5508X^4-5832X^5+2187X^6,\\
U(X^3)&=X- 360X^2+ 10566X^3- 99144X^4+ 423549X^5- 944784X^6\notag\\
&\qquad + 1141614X^7 -708588X^8 +177147X^9.\label{eq:UX3}
\end{align}
\end{theorem}

\begin{proof}
We have
\begin{align*}
X&=\frac{\phi(-q^3)^4}{\phi(-q)^4}=\phi(-q^3)^4\left(\frac{\phi(-q^9)^3}{\phi(-q^3)^4}\left(1+2qw(q^3)+4q^2 w(q^3)^2\right)\right)^4\\
&=\frac{\phi(-q^9)^{12}}{\phi(-q^3)^{12}}\big(1 + 8 q w(q^3) + 40 q^2 w(q^3)^2 + 128 q^3 w(q^3)^3 + 304 q^4 w(q^3)^4\\
&\quad\quad + 512 q^5 w(q^3)^5 + 
 640 q^6 w(q^3)^6 + 512 q^7 w(q^3)^7 + 256 q^8 w(q^3)^8\big).
\end{align*}
It therefore follows that
\begin{align*}
U(X)=\frac{\phi(-q^3)^{12}}{\phi(-q)^{12}}\big(1 + 128 q w(q)^3 + 640 q^2 w(q)^6 \big).
\end{align*}
Applying Lemma \ref{le:qw3}, we have
\begin{align*}
U(X)&=\frac{\phi(-q^3)^{12}}{\phi(-q)^{12}}\left(1 + 128 \left(\frac{1}{8}\left(1-\frac{\phi(-q)^4}{\phi(-q^3)^4}\right)\right) + 640 \left(\frac{1}{8}\left(1-\frac{\phi(-q)^4}{\phi(-q^3)^4}\right)\right)^2 \right)\\
&=10\frac{\phi(-q^3)^4}{\phi(-q)^4}-36\left(\frac{\phi(-q^3)^4}{\phi(-q)^4}\right)^2+27\left(\frac{\phi(-q^3)^4}{\phi(-q)^4}\right)^3\\
&=10X- 36X^2+ 27X^3.
\end{align*}

Similarly, we have
\begin{align*}
U(X^2)&=\frac{\phi(-q^3)^{24}}{\phi(-q)^{24}}\big(1 + 896 q w(q)^3 + 50176 q^2 w(q)^6 + 520192 q^3 w(q)^9\\
&\quad\quad + 1089536 q^4 w(q)^{12} + 262144 q^5 w(q)^{15} \big)\\
&=\frac{\phi(-q^3)^{24}}{\phi(-q)^{24}}\Bigg(1 + 896 \left(\frac{1}{8}\left(1-\frac{\phi(-q)^4}{\phi(-q^3)^4}\right)\right) + 50176 \left(\frac{1}{8}\left(1-\frac{\phi(-q)^4}{\phi(-q^3)^4}\right)\right)^2\\
&\quad\quad  + 520192 \left(\frac{1}{8}\left(1-\frac{\phi(-q)^4}{\phi(-q^3)^4}\right)\right)^3 + 1089536 \left(\frac{1}{8}\left(1-\frac{\phi(-q)^4}{\phi(-q^3)^4}\right)\right)^4\\
&\quad\quad   + 262144 \left(\frac{1}{8}\left(1-\frac{\phi(-q)^4}{\phi(-q^3)^4}\right)\right)^5 \Bigg)\\
&=-8X+306X^2-2160X^3+5508X^4-5832X^5+2187X^6.
\end{align*}

The proof of \eqref{eq:UX3} follows in the same way.
\end{proof}

We also have
\begin{theorem}\label{th:UxiX}
It holds that
\begin{align}
U(\xi)&=9X,\label{eq:xi1}\\
U(\xi X)&=-9X+ 252X^2- 891X^3+ 729X^4,\\
U(\xi X^2)&=X -378X^2+ 8613X^3 -54675X^4 +138510X^5-150903X^6\notag\\
&\qquad+ 59049X^7,\label{eq:UxiX2}\\
U(\xi X^3)&=147X^2 -14553X^3 +312255X^4 -2617839X^5 +10764414X^6\notag\\
&\qquad -23914845X^7 +29288304X^8 -18600435X^9 + 4782969X^{10}.\label{eq:UxiX3}
\end{align}
\end{theorem}

\begin{proof}
We have
\begin{align*}
\xi&=q^{-2}\frac{\psi(q)^2 \phi(-q^9)}{\phi(-q) \psi(q^9)^2}\\
&=q^{-2}\frac{\phi(-q^9)^{4}}{\phi(-q^3)^{4}}\left(1+2qw(q^3)+4q^2 w(q^3)^2\right)\left(\frac{1}{w(q^3)}+q\right)^2\\
&=\frac{\phi(-q^9)^{4}}{\phi(-q^3)^{4}}\left(\frac{1}{q^2 w(q^3)^2}+\frac{4}{q w(q^3)}+9+10q w(q^3)+4q^2 w(q^3)^2\right).
\end{align*}
It follows that
\begin{align*}
U(\xi)=9\frac{\phi(-q^3)^{4}}{\phi(-q)^{4}}=9X.
\end{align*}

Similarly, we have
\begin{align*}
U(\xi X)&=\frac{\phi(-q^3)^{16}}{\phi(-q)^{16}}\left(81 + 3312 q w(q)^3 + 14400 q^2 w(q)^6 + 4608 q^3 w(q)^9\right)\\
&=\frac{\phi(-q^3)^{16}}{\phi(-q)^{16}}\Bigg(81 + 3312 \left(\frac{1}{8}\left(1-\frac{\phi(-q)^4}{\phi(-q^3)^4}\right)\right)\\
&\quad\quad + 14400 \left(\frac{1}{8}\left(1-\frac{\phi(-q)^4}{\phi(-q^3)^4}\right)\right)^2  + 4608 \left(\frac{1}{8}\left(1-\frac{\phi(-q)^4}{\phi(-q^3)^4}\right)\right)^3 \Bigg)\\
&=-9X+ 252X^2- 891X^3+ 729X^4.
\end{align*}

The proofs of \eqref{eq:UxiX2} and \eqref{eq:UxiX3} follow in the same way.
\end{proof}

\begin{remark}
Here \eqref{eq:xi1} is also a direct consequence of a result of Chern and Tang; see \cite[Corollary 2.7]{CT2018}.
\end{remark}

\subsection{Further relations}

We now define the following two matrices $(a_{i,j})_{i\ge 1,\, j\ge 1}$ and $(b_{i,j})_{i\ge 1,\, j\ge 1}$ (for all matrices defined here and below, we sometimes use $m(i,j)$ to denote $m_{i,j}$):

\begin{enumerate}[(1).]
\item For $1\le i\le 3$, the entries $a(i,j)$ and $b(i,j)$ are respectively given by
\begin{align*}
U(X^i)&=\sum_{j\ge 1}a(i,j)X^j,\\
U(\xi X^i)&=\sum_{j\ge 1}b(i,j)X^j.
\end{align*}
We refer to Theorems \ref{th:UX} and \ref{th:UxiX} for their exact values.
\item For $i\ge 4$, both matrices satisfy the following recurrence relation (here $m$ stands for a matrix):
\begin{align*}
m(i,j)=&30m(i-1,j-1)-108m(i-1,j-2)+81m(i-1,j-3)\\
&-12m(i-2,j-1)+9m(i-2,j-2)+m(i-3,j-1).
\end{align*}
%
\end{enumerate}
Note that all unspecified and undefined entries in a matrix are assumed to be $0$. 

\begin{theorem}
For $i\ge 1$, we have
\begin{align}
U(X^i)&=\sum_{j\ge 1}a(i,j)X^j,\label{eq:rec:a}\\
U(\xi X^i)&=\sum_{j\ge 1}b(i,j)X^j.\label{eq:rec:b}
\end{align}
\end{theorem}

\begin{proof}
According to the definition of $(a_{i,j})$ and $(b_{i,j})$, the theorem holds for $1\le i\le 3$. We now deal with the cases $i\ge 4$.

Let
$$x_t=X\left(\frac{\tau+t}{3}\right),\quad t=1,\ldots,3.$$
Let $\sigma_t$ ($t=1,\ldots,3$) be the $t$-th elementary symmetric function of $x_1,\ldots,x_3$, i.e.,
\begin{align*}
\sigma_1&=x_1+x_2+x_3,\\
\sigma_2&=x_1x_2+x_2x_3+x_3x_1,\\
\sigma_3&=x_1x_2x_3.
\end{align*}
Then each $x_t$ ($t=1,\ldots,3$) satisfies
$$x_t^3-\sigma_1 x_t^2+\sigma_2 x_t-\sigma_3=0.$$
For $i\ge 1$, we put $p_i=x_1^i+x_2^i+x_3^i$.

Note that for any integer $i$ we have
$$3U(X^i)=\sum_{t=1}^3 X\left(\frac{\tau+t}{3}\right)^i=p_i.$$
By Theorem \ref{th:UX}, we can write $p_1,\ldots,p_3$ as polynomials of $X$. It follows from Newton's identities (cf.~\cite{Mea1992}) that
\begin{align*}
\sigma_1&=p_1=30 X - 108 X^2 + 81 X^3,\\
\sigma_2&=(\sigma_1 p_1-p_2)/2=12 X - 9 X^2,\\
\sigma_3&=(\sigma_2 p_1-\sigma_1 p_2+p_3)/3=X.
\end{align*}

Let $u(\tau):\mathbb{H}\to\mathbb{C}$ be a complex function. We put $u_t=u\left(\frac{\tau+t}{3}\right)$ for $t=1,\ldots,3$. Then for $i\ge 4$,
\begin{align*}
3U(u X^i)&=\sum_{t=1}^3 u_t x_t^i=\sum_{t=1}^3 u_t (\sigma_1 x_t^{i-1}-\sigma_2 x_t^{i-2}+\sigma_3 x_t^{i-3})\\
&=\sigma_1\sum_{t=1}^3 u_t x_t^{i-1}-\sigma_2\sum_{t=1}^3 u_t x_t^{i-2}+\sigma_3\sum_{t=1}^3 u_t x_t^{i-3}\\
&= 3\left(\sigma_1 U(u X^{i-1})- \sigma_2 U(u X^{i-2}) + \sigma_3U(u X^{i-3})\right).
\end{align*}

At last, we respectively take $u=1$ and $u=\xi$ to arrive at the desired result.
\end{proof}

Let $(d_\alpha)_{\alpha\ge 1}$ be a family of integer sequences with
$$d_1=(9,0,0,\ldots).$$
For $\alpha\ge 2$, we recursively define
\begin{align*}
d_{\alpha}(j)=\begin{cases}
\sum_{k\ge 1} a(k,j)d_{\alpha-1}(k) & \text{if $\alpha$ is even},\\
\sum_{k\ge 1} b(k,j)d_{\alpha-1}(k) & \text{if $\alpha$ is odd},
\end{cases}
\end{align*}
where $d_{\alpha}(j)$ denotes the $j$-th element of sequence $d_\alpha$.

Let
\begin{equation}
\sum_{n\ge 0}g(n)q^n=\frac{E(q^2)^5}{E(q)^4}.
\end{equation}

We have
\begin{theorem}\label{th:g-d}
For $\alpha\ge 1$, we have
\begin{align}
\sum_{n\ge 0} g\left(3^{2\alpha-1}n+\frac{3^{2\alpha}-1}{4}\right)q^n&=\frac{E(q^{6})^5}{E(q^3)^4}\sum_{j\ge 1}d_{2\alpha-1}(j)X^j,\\
\sum_{n\ge 0} g\left(3^{2\alpha}n+\frac{3^{2\alpha}-1}{4}\right) q^n&=\frac{E(q^{2})^5}{E(q)^4}\sum_{j\ge 1}d_{2\alpha}(j)X^j.
\end{align}
\end{theorem}

\begin{proof}
Theorem \ref{th:UxiX} tells us $U(\xi)=9X=\sum_{j\ge 1}d_1(j)X^j$. On the other hand, we have
\begin{align*}
U(\xi)&=U\left(q^{-2}\frac{\ee{2}^5 \ee{9}^4}{\ee{1}^4 \ee{18}^5}\right)=\frac{\ee{3}^4}{\ee{6}^5}U\left(\sum_{n\ge 0} g(n) q^{n-2}\right)\\
&=\frac{\ee{3}^4}{\ee{6}^5}\sum_{n\ge 0} g(3n+2) q^{n}.
\end{align*}
Theorem \ref{th:g-d} is therefore valid for $\alpha=1$.

We now assume that the theorem holds for some odd positive integer $2\alpha-1$. Then we have
$$\sum_{j\ge 1}d_{2\alpha-1}(j)X^j=\frac{\ee{3}^4}{\ee{6}^5} \sum_{n\ge 0} g\left(3^{2\alpha-1}n+\frac{3^{2\alpha}-1}{4}\right)q^n.$$

We now apply the $U$-operator to both sides of the above identity. Then the l.h.s.~ becomes
\begin{align*}
U\left(\sum_{j\ge 1}d_{2\alpha-1}(j)X^j\right)&=\sum_{j\ge 1}d_{2\alpha-1}(j)U(X^j)=\sum_{j\ge 1}d_{2\alpha-1}(j) \sum_{\ell\ge 1}a(j,\ell)X^\ell\\
&=\sum_{\ell\ge 1}\left(\sum_{j\ge 1}a(j,\ell) d_{2\alpha-1}(j)\right)X^\ell=\sum_{\ell\ge 1}d_{2\alpha}(\ell)X^\ell.
\end{align*}
On the other hand, the r.h.s.~is
\begin{align*}
U\left(\frac{\ee{3}^4}{\ee{6}^5} \sum_{n\ge 0} g\left(3^{2\alpha-1}n+\frac{3^{2\alpha}-1}{4}\right)q^n\right)= \frac{\ee{1}^4}{\ee{2}^5}\sum_{n\ge 0} g\left(3^{2\alpha}n+\frac{3^{2\alpha}-1}{4}\right) q^n.
\end{align*}
Combining them together, we have
$$\sum_{j\ge 1}d_{2\alpha}(j)X^j=\frac{\ee{1}^4}{\ee{2}^5}\sum_{n\ge 0} g\left(3^{2\alpha}n+\frac{3^{2\alpha}-1}{4}\right) q^n.$$

We next multiply by $\xi$ on both sides of the above identity and then apply the $U$-operator. Then the l.h.s.~ becomes
\begin{align*}
U\left(\sum_{j\ge 1}d_{2\alpha}(j)\xi X^j\right)&=\sum_{j\ge 1}d_{2\alpha}(j)U(\xi X^j)=\sum_{j\ge 1}d_{2\alpha}(j) \sum_{\ell\ge 1}b(j,\ell)X^\ell\\
&=\sum_{\ell\ge 1}\left(\sum_{j\ge 1}b(j,\ell) d_{2\alpha}(j)\right)X^\ell=\sum_{\ell\ge 1}d_{2\alpha+1}(\ell)X^\ell.
\end{align*}
On the other hand, the r.h.s.~is
\begin{align*}
&U\left(q^{-2}\frac{\ee{9}^4}{\ee{18}^5} \sum_{n\ge 0} g\left(3^{2\alpha}n+\frac{3^{2\alpha}-1}{4}\right)q^n\right)\\
&\qquad= \frac{\ee{3}^4}{\ee{6}^5}\sum_{n\ge 0} g\left(3^{2\alpha+1}n+\frac{3^{2\alpha+2}-1}{4}\right) q^n.
\end{align*}
Hence
$$\sum_{j\ge 1}d_{2\alpha+1}(j)X^j=\frac{\ee{3}^4}{\ee{6}^5}\sum_{n\ge 0} g\left(3^{2\alpha+1}n+\frac{3^{2\alpha+2}-1}{4}\right) q^n.$$

The theorem therefore follows by induction.
\end{proof}

At last, we notice that the two infinite matrices $(a_{i,j})$ and $(b_{i,j})$ are row and column finite. Define $(t_{i,j})_{i\ge 1,\, j\ge 1}$ by
$$t_{i,j}=\sum_{k\ge 1} a_{i,k}b_{k,j}.$$
Then the above sum is indeed a finite sum. From the recurrence relation for $d_\alpha$, one has, for $\alpha\ge 1$,
\begin{equation}\label{eq:2a+1}
d_{2\alpha+1}(j)=\sum_{i\ge 1}t(i,j)d_{2\alpha-1}(i).
\end{equation}

\section{The $3$-adic orders}

\subsection{The $3$-adic orders}

For any integer $n$, let $\pi(n)$ be the $3$-adic order of $n$ with the convention that $\pi(0)=\infty$. Let $\fl{x}$ be the largest integer not exceeding $x$.

We begin with the $3$-adic orders of $(a_{i,j})$ and $(b_{i,j})$.

\begin{theorem}
For $i$ and $j\ge 1$, it holds that
\begin{align}
\pi(a(i,j))&\ge \fl{\frac{3j-i-1}{2}},\label{ineq:a}\\
\pi(b(i,j))&\ge \fl{\frac{3j-i}{2}}.\label{ineq:b}
\end{align}
\end{theorem}

\begin{proof}
One may first check \eqref{ineq:a} directly for $i=1,\ldots,3$. Assume that \eqref{ineq:a} holds for $1,\ldots,i-1$ with some $i\ge 4$. We have (for some undefined entries like $a(i,0)$, since we assign its value to be $0$, its $3$-adic order is therefore $\infty$):
\begin{align*}
\pi(a(i-1,j-1))+1\ge \fl{\dfrac{3(j-1)-(i-1)-1}{2}}+1&\ge \fl{\dfrac{3j-i-1}{2}},\\
\pi(a(i-1,j-2))+3\ge \fl{\dfrac{3(j-2)-(i-1)-1}{2}}+3&\ge \fl{\dfrac{3j-i-1}{2}},\\
\pi(a(i-1,j-3))+4\ge \fl{\dfrac{3(j-3)-(i-1)-1}{2}}+4&\ge \fl{\dfrac{3j-i-1}{2}},\\
\pi(a(i-2,j-1))+1\ge \fl{\dfrac{3(j-1)-(i-2)-1}{2}}+1&\ge \fl{\dfrac{3j-i-1}{2}},\\
\pi(a(i-2,j-2))+2\ge \fl{\dfrac{3(j-2)-(i-2)-1}{2}}+2&\ge \fl{\dfrac{3j-i-1}{2}},\\
\pi(a(i-3,j-1))+0\ge \fl{\dfrac{3(j-1)-(i-3)-1}{2}}+0&\ge \fl{\dfrac{3j-i-1}{2}}.
\end{align*}
It follows from the recurrence relation of $(a_{i,j})$ that $\pi(a(i,j))$ is at least the minimum of the l.h.s.~of the above 6 inequalities. We hence obtain \eqref{ineq:a} by induction.

\eqref{ineq:b} follows in the same way.
\end{proof}

As a direct consequence, we have

\begin{corollary}
For $i$ and $j\ge 1$, it holds that
\begin{equation}\label{ineq:t}
\pi(t(i,j))\ge \min_{k\ge 1}\left\{\pi(a(i,k))+\pi(b(k,j))\right\}\ge \min_{k\ge 1}\left\{\fl{\frac{3k-i-1}{2}}+\fl{\frac{3j-k}{2}}\right\}.
\end{equation}
\end{corollary}

On the other hand, one may compute the values of $(\pi(t(i,j)))_{i\ge 1,\, j\ge 1}$ for some small $i$ and $j$, which are listed below
$$\begin{pmatrix}
2 & 2 & 4 & \cdots\\
3 & 2 & 4 & \cdots\\
2 & 2 & 4 & \cdots\\
0 & 3 & 3 & \cdots\\
0 & 2 & 3 & \cdots\\
\vdots & \vdots & \vdots & \ddots
\end{pmatrix}_{i\ge 1,\, j\ge 1}.$$

We next show

\begin{theorem}\label{th:3order}
For $\alpha$ and $j\ge 1$, it holds that
\begin{align}
\pi(d_{2\alpha-1}(j))\ge 2\alpha+\fl{\frac{2j-2}{3}}.\label{ineq:d}
\end{align}
\end{theorem}

\begin{proof}
Since $d_1=(9,0,0,\ldots)$, we see that \eqref{ineq:d} holds for $\alpha=1$. Assume that the theorem is valid for some $\alpha\ge 1$. We want to show
$$\pi(d_{2\alpha+1}(j))\ge 2\alpha+2+\fl{\frac{2j-2}{3}}.$$
Then the theorem follows by induction.

It follows from \eqref{eq:2a+1} that
\begin{align*}
\pi(d_{2\alpha+1}(j))\ge \min_{i\ge 1}\left\{\pi(d_{2\alpha-1}(i))+\pi(t(i,j))\right\}.
\end{align*}

We further know from \eqref{ineq:t} that
\begin{align*}
\pi(d_{2\alpha-1}(i))+\pi(t(i,j))&\ge 2\alpha+\fl{\frac{2i-2}{3}}+\min_{k\ge 1}\left\{\fl{\frac{3k-i-1}{2}}+\fl{\frac{3j-k}{2}}\right\}\\
&\ge 2\alpha+\fl{\frac{2i-2}{3}}+\min_{k\ge 1}\left\{\fl{\frac{3j+2k-i-3}{2}}\right\}\\
&=2\alpha+\fl{\frac{2i-2}{3}}+\fl{\frac{3j-i-1}{2}}\\
&\ge 2\alpha+2+\fl{\frac{2j-2}{3}}+\fl{\frac{5j+i-21}{6}}.
\end{align*}
Note that $i\ge 1$. Hence for $j\ge 4$, we have $5j+i\ge 21$. Hence we merely need to check the cases $j=1,\ldots,3$.

When $j=1$ or $2$, we need to show
$$\pi(d_{2\alpha+1}(j))\ge 2\alpha+2.$$
Note that
$$\pi(d_{2\alpha-1}(i))+\pi(t(i,j))\ge 2\alpha+\fl{\frac{2i-2}{3}}+\pi(t(i,j)).$$
We have $\fl{(2i-2)/3}\ge 0$ and $\pi(t(i,j))\ge 2$ when $1\le i\le 3$ and $j=1$ or $2$, and $\fl{(2i-2)/3}\ge 2$ and $\pi(t(i,j))\ge 0$ when $i\ge 4$. Hence for $j=1$ or $2$
$$\pi(d_{2\alpha+1}(j))\ge \min_{i\ge 1}\left\{\pi(d_{2\alpha-1}(i))+\pi(t(i,j))\right\}\ge 2\alpha+2.$$

When $j=3$, we need to show
$$\pi(d_{2\alpha+1}(3))\ge 2\alpha+3.$$
Note again that
$$\pi(d_{2\alpha-1}(i))+\pi(t(i,3))\ge 2\alpha+\fl{\frac{2i-2}{3}}+\pi(t(i,3)).$$
We have $\fl{(2i-2)/3}\ge 0$ and $\pi(t(i,3))\ge 3$ when $1\le i\le 5$, and $\fl{(2i-2)/3}\ge 3$ and $\pi(t(i,3))\ge 0$ when $i\ge 6$. Hence
$$\pi(d_{2\alpha+1}(3))\ge \min_{i\ge 1}\left\{\pi(d_{2\alpha-1}(i))+\pi(t(i,3))\right\}\ge 2\alpha+3.$$

We arrive at the desired result.
\end{proof}

\subsection{Proof of the main theorem}

It follows from Theorems \ref{th:g-d} and \ref{th:3order} that
\begin{align*}
\sum_{n\ge 0} g\left(3^{2\alpha-1}n+\frac{3^{2\alpha}-1}{4}\right)q^n&=\frac{E(q^{6})^5}{E(q^3)^4}\sum_{j\ge 1}d_{2\alpha-1}(j)X^j\equiv 0 \pmod{3^{2\alpha}},
\end{align*}
since $\pi(d_{2\alpha-1}(j))\ge 2\alpha$ for all $j$. Hence we have

\begin{theorem}
For $\alpha\ge 1$ and $n\ge 0$, we have
\begin{equation}
g\left(3^{2\alpha-1}n+\frac{3^{2\alpha}-1}{4}\right)\equiv 0 \pmod{3^{2\alpha}}.
\end{equation}
\end{theorem}

At last, we notice from \eqref{eq:a2n} that
$$\sum_{n\ge 0}\bb(2n) q^n =\frac{E(q^2)^5}{E(q)^4} = \sum_{n\ge 0}g(n)q^n.$$
The desired congruence \eqref{eq:cong} therefore follows.

\section{An alternative proof of (\ref{even})}

We end this paper with a more direct proof of \eqref{even}. Recall that Ramanujan's bilateral ${}_1\psi_1$ identity \cite[Eq.~(5.2.1)]{GR1990} tells us
\begin{equation}\label{eq:1psi1}
\sum_{n=-\infty}^\infty \frac{(a;q)_n z^n}{(b;q)_n} = \frac{(q,b/a,az,q/az;q)_\infty}{(b,q/a,z,b/az;q)_\infty},\quad |b/a|<|z|<1.
\end{equation}

We now replace $q$ by $q^2$ and take $a=-q$, $b=q^3$, $z=q$ in \eqref{eq:1psi1}. Then
$$\sum_{n=-\infty}^\infty \frac{(-q;q^2)_n q^n}{(q^3;q^2)_n} = \frac{(q^2,-q^2,-q^2,-1;q^2)_\infty}{(q^3,-q,q,-q;q^2)_\infty}.$$

Note that
\begin{align*}
\sum_{n=-\infty}^\infty \frac{(-q;q^2)_n q^n}{(q^3;q^2)_n}&= \sum_{n\ge 0} \frac{(-q;q^2)_n q^n}{(q^3;q^2)_n}+\sum_{n\ge 1} \frac{(-q;q^2)_{-n} q^{-n}}{(q^3;q^2)_{-n}}\\
&=(1-q)\left(\sum_{n\ge 0}\frac{(-q;q^2)_nq^n}{(q;q^2)_{n+1}}+\sum_{n\ge 0}\frac{(q;q^2)_n(-q)^n}{(-q;q^2)_{n+1}}\right).
\end{align*}
On the other hand,
$$\frac{(q^2,-q^2,-q^2,-1;q^2)_\infty}{(q^3,-q,q,-q;q^2)_\infty}=2(1-q)\frac{(q^4;q^4)_\infty^5}{(q^2;q^2)_{\infty}^4}.$$

It follows that
$$\sum_{n\ge 0}\frac{(-q;q^2)_nq^n}{(q;q^2)_{n+1}}+\sum_{n\ge 0}\frac{(q;q^2)_n(-q)^n}{(-q;q^2)_{n+1}}
=2\frac{(q^4;q^4)_\infty^5}{(q^2;q^2)_{\infty}^4},$$
which is essentially
$$\beta(q)+\beta(-q)=2\frac{(q^4;q^4)_\infty^5}{(q^2;q^2)_{\infty}^4}.$$

\subsection*{Acknowledgements}

We want to thank Ae Ja Yee for some helpful discussions. C.~Wang was partially supported by the outstanding doctoral dissertation cultivation plan of action (No.~YB2016028).

\bibliographystyle{amsplain}

\end{document}